\documentclass[a4paper,11pt]{amsart}
\usepackage{amssymb,amsthm,bbm,a4wide,graphics,verbatim}

\DeclareMathOperator{\GL}{\mathrm{GL}}
\DeclareMathOperator{\Irr}{\mathrm{Irr}}

\DeclareMathOperator{\End}{\mathrm{End}}

\DeclareMathOperator{\gr}{\mathrm{gr}}
\DeclareMathOperator{\Ad}{\mathrm{Ad}}

\def\ii{\mathrm{i}}

\def\C{\mathbbm{C}}

\def\Q{\mathbbm{Q}}
\def\Z{\mathbbm{Z}}
\def\la{\lambda}

\newtheorem{prop}{Proposition}[section]
\newtheorem{lemma}[prop]{Lemma}
\newtheorem{cor}[prop]{Corollary}
\newtheorem{theo}[prop]{Theorem}

\newtheorem{conj}[prop]{Conjecture}

\renewcommand{\k}{\mathbbm{k}}
\def\un{\mathbbm{1}}
\newcommand{\into}{\hookrightarrow}
\newcommand{\onto}{\twoheadrightarrow}

 3
 2

\def\cnode#1{{\kern -0.4pt\bigcirc\kern-7pt{\scriptstyle#1}\kern 2.6pt}}
\def\ncnode#1#2{{\kern -0.4pt\mathop\bigcirc\limits_{#2}\kern-8.6pt{\scriptstyle#1}\kern 2.3pt}}
\def\node{{\kern -0.4pt\bigcirc\kern -1pt}}
\def\snode{{\kern -0.4pt{\scriptstyle\bigcirc}\kern -1pt}}
\def\nnode#1{{\kern -0.6pt\mathop\bigcirc\limits_{#1}\kern -1pt}}
\def\bar#1pt{{\vrule width#1pt height3pt depth-2pt}}
\def\dbar#1pt{{\rlap{\vrule width#1pt height2pt depth-1pt} 
                 \vrule width#1pt height4pt depth-3pt}}
\def\tbar#1pt{{\rlap{\vrule width#1pt height1pt depth0pt}
           \rlap{\vrule width#1pt height3pt depth-2pt}
\vrule width#1pt height5pt depth-4pt}}
\def\overmark#1#2{\kern -1.5pt\mathop{#2}\limits^{#1}\kern -2pt}
\def\trianglerel#1#2#3{
 \nnode#1\bar14pt\kern-13pt\raise7.5pt\hbox{$\displaystyle \underleftarrow 6$}
  \kern 2pt\nnode#2
 \kern-29pt\raise9.5pt\hbox{$\diagup$}
 \kern -2pt
 \raise17.5pt\hbox{$\node$\rlap{\raise 2pt\hbox{$\kern 1pt\scriptstyle #3$}}}
 \kern -1pt\raise9.5pt\hbox{$\diagdown$}
 \kern 3pt
}
\def\vertbar#1#2{\rlap{$\nnode{#1}$}
                 \rlap{\kern4pt\vrule width1pt height17.3pt depth-7.3pt}
\raise19.4pt\hbox{$\node$\rlap{$\kern 1pt\scriptstyle#2$}}}

\def\SN{\mathfrak{S}_n}
\DeclareMathOperator{\Ind}{\mathrm{Ind}}
\DeclareMathOperator{\Res}{\mathrm{Res}}

\title{Hooks generate the representation ring of the symmetric group}
\author{Ivan Marin}
\date{October 7, 2010}
\begin{document}

\maketitle

\bigskip
\begin{center}
Institut de Math\'ematiques de Jussieu \\
Universit\'e Paris 7 \\
175 rue du Chevaleret \\
F-75013 Paris
\end{center}
\bigskip

\noindent {\bf Abstract.} We prove that the representation ring of the
symmetric group on $n$ letters is generated by the exterior powers of its natural
$(n-1)$-dimensional representation. The proof we give illustrates a strikingly
simple formula due to Y. Dvir. We provide an application and investigate a
possible generalization of this result to some other reflection groups.

\bigskip


\bigskip

%20C30 Representations of finite symmetric groups
%20F36 Braid groups, Artin groups
%57R22 Topology of vector bundles and fiber bundles

\section{Introduction}

We let $\mathfrak{S}_n$ denote the symmetric group on $n$ letters,
and $R(\mathfrak{S}_n)$ its ordinary representation ring,
or equivalently the ring of its complex characters. It is a free $\Z$-module
with basis $(V_{\la})_{\la \vdash n}$ of irreductible characters classically
indexed by the set of partitions $\la = [\la_1 ,\la_2, \dots]$ of $n = \la_1 + \la_2 + \dots$
with $\la_1 \geq \la_2 \geq \dots \geq 0$. As usual, we identify such a $\la$
with a Young (or Ferrers) diagram, and we use the
row-aligned, left-justified, top-to-bottom convention (e.g. the
left-hand sides of figure \ref{figurerestretinduit} represent the
partition $[3,2,1,1]$). The \emph{size} $n$ of the partition $\la$ is
denoted $|\la|$.

We refer to \cite{FH} for classical facts about the correspondence between
representations and partitions.
The notation we use here is such that the partition $[n]$ is attached to
the trivial representation $V_{[n]} = \un$, and the natural permutation representation
$\SN < \GL_n(\C)$ decomposes itself as $\C^n = \un + V$ with $V = V_{[n-1,1]}$.
Among the classical results that can be found in \cite{FH} we recall that
the exterior powers $\Lambda^k V$ for $0 \leq k \leq n$
provide irreducible representations attached to the partitions $[n-k,1^k]$.
Such representations or the corresponding partitions are classically called \emph{hooks}.

The purpose of this note is to prove the following.

\begin{theo} \label{maintheo} For every $n \geq 1$, the representation ring $R(\SN)$
is generated by the hooks $\Lambda^k V, 0 \leq k \leq n-1$.
\end{theo}

Note that $\Lambda^{k+1} \C^n = \Lambda^{k+1} V \oplus \Lambda^k V$, hence
the collection of the $\Lambda^k V$ and the collection of the $\Lambda^k \C^n$ span the
same additive subgroup of $R(\SN)$.
Another version of the same result is thus the following.
 
\begin{theo} For every $n \geq 1$, the representation ring $R(\SN)$
is generated by the representations $\Lambda^k \C^n, 0 \leq k \leq n-1$.
\end{theo}

This latter version can be compared with the similar classical result for
$\GL_n(\C)$, that its ring of rational representations is
generated by the $\Lambda^k \C^n$ (which, in terms of characters,
simply means that the symmetric polynomials are generated by the
elementary symmetric ones -- see e.g. \cite{FH}, (6.2) and appendix A).

It has been communicated to us by J.-Y. Thibon that, when translated
in the language of symmetric functions, the theorems above are equivalent
to the results of Butler and Boorman (see \cite{BUTLER,BOORMAN} and
also \cite{STW}). The main point of this note is thus to show how
to derive this result from the strikingly simple formula of Dvir
(see \S 3 below), and to explore natural generalizations.

It is indeed a remarkable fact that, while tensor product
decompositions are very well-understood for the representations
of reductive Lie groups, the ring structure of classical finite
groups is often difficult to understand in terms of the natural
indexing of their irreducible representations. Having a nice
generating family for its representationg ring is typically one of the nice features of the symmetric
group that one would like to generalize.

\section{A filtration on $R(\SN)$}

Let $G$ be a finite group, $V$ a faithful (finite-dimensional,
complex, linear) representation of $G$ and $\Irr(G)$ the set of all irreducible
representations. Then the representation ring $R(G)$ is a free $\Z$-module
with basis $\Irr(G)$, and each $\rho \in \Irr(G)$ embeds into some
$V^{\otimes r}$ for $r \in \Z_{\geq 0}$ (Burnside-Molien, see e.g. \cite{FH} problem 2.37).
The \emph{level} (or \emph{depth}) of $\rho \in \Irr(G)$ with respect to $V$ is defined to be
$$
N(\rho) = \min \{ r \in \Z_{\geq 0} \ | \ \rho \into V^{\otimes r} \}
$$
Obviously we have $N(\rho_1 \otimes \rho_2) \leq N(\rho_1) + N(\rho_2)$,
$N(\un) = 0$. It follows that the subgroup $\mathcal{F}_r$ of $R(G)$
generated by the $\rho \in \Irr(G)$ with $N(\rho) \leq r$
defines a ring filtration $\mathcal{F}_0 \subset \mathcal{F}_1 \subset
\dots \subset$ of $R(G)$, hence a ring structure $(\gr R(G), \odot)$
on the graded ring
$$
\gr R(G) = \bigoplus_{k=0}^{+\infty} (\mathcal{F}_{k} R(G))/(\mathcal{F}_{k-1} R(G))
$$
with the convention $\mathcal{F}_{-1} R(G) = \{ 0 \}$. Notice that
$\Irr(G)$ provides a basis of $R(G)$ as a $\Z$-module.

We now let $G = \SN$.
% and $V$ the irreducible representation
%attached to $[n-1,1]$. 
Considering $\mathfrak{S}_{n-1} < \SN$
through the natural embedding that leaves the $n$-th letter untouched,
we let $\Ind : R(\mathfrak{S}_{n-1}) \to R(\SN)$
and $\Res : R(\SN) \to R(\mathfrak{S}_{n-1})$ denote the usual
induction and restriction morphisms. 
\begin{figure}
\begin{center}
\resizebox{12cm}{!}{
\includegraphics{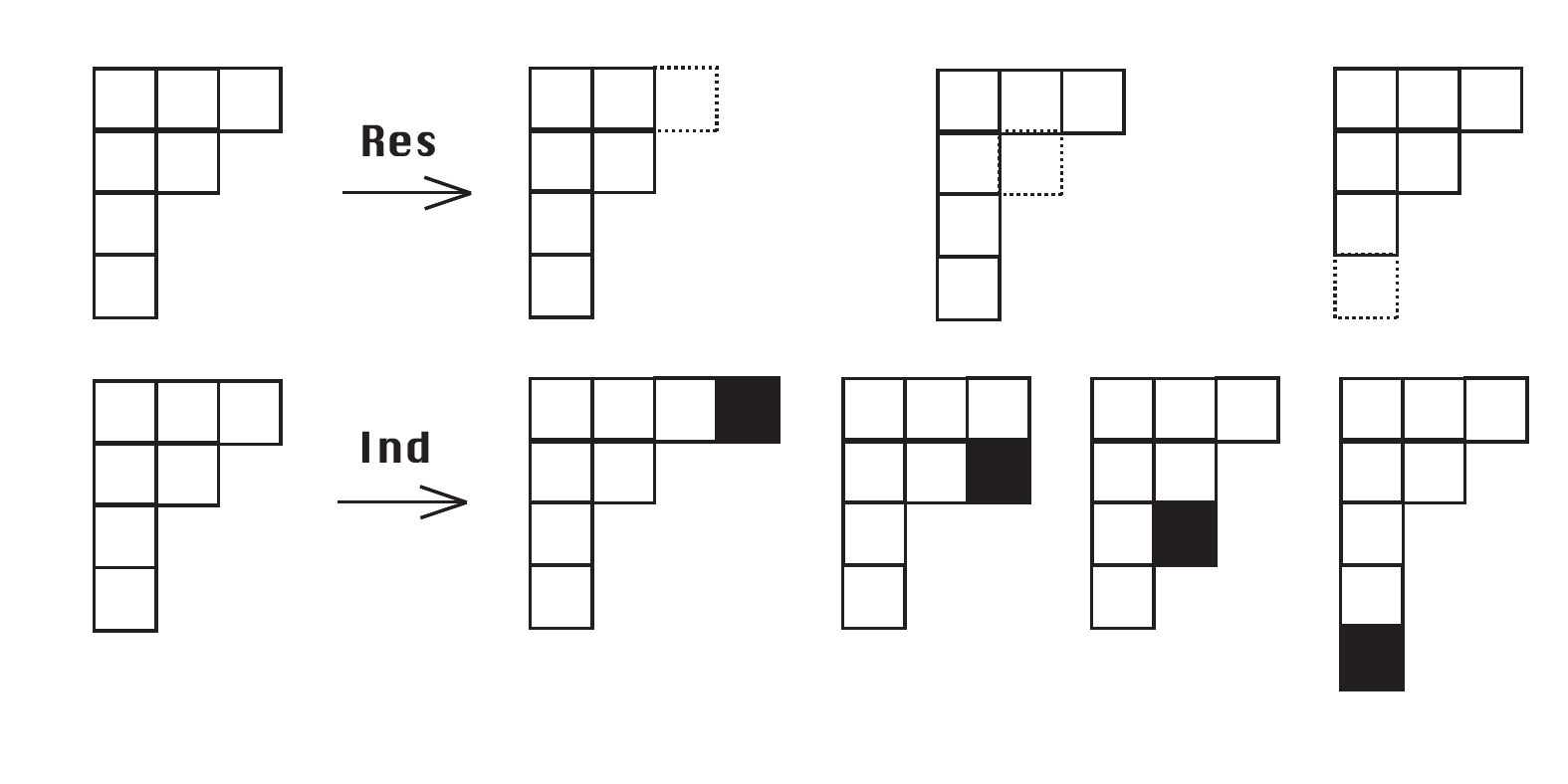}}
\end{center} 
\caption{Restriction and induction on Young diagrams}
\label{figurerestretinduit}
\end{figure}

Recall that $\Res$ and $\Ind$ are easily described on Young diagrams
by Young rule, as illustrated by figure \ref{figurerestretinduit}. If $\la$ is a Young diagram of size $n$, then $\Res V_{\la}$
%(respectively $\Ind V_{\la}$) 
is the sum (without multiplicities) in $R(\mathfrak{S}_{n-1})$
of the $V_{\mu}$, with 
%(the Young diagram of) 
$\mu$ being deduced from $\la$ by removing 
(respectively adding) 
one box.
Similarly, if $\la$ is a Young diagram of size $n-1$, then $\Ind V_{\la}$
is the sum (without multiplicities) in $R(\mathfrak{S}_{n})$
of the $V_{\mu}$, with 
%(the Young diagram of) 
$\mu$ being deduced from $\la$ by adding  one box.

The operator $\Ind \Res$ on Young diagrams then means summing all
$V_{\mu}$ for $\mu$ a diagram deduced from $\la$ by \emph{moving} one box, and
$\delta(\la)$ copies of $V_{\la}$ where $\delta(\la) = \# \{ i \ | \ \la_i
\neq \la_{i+1} \}$ (see figure \ref{figureindres}).

\begin{figure}
\begin{center}
\resizebox{8cm}{!}{\includegraphics{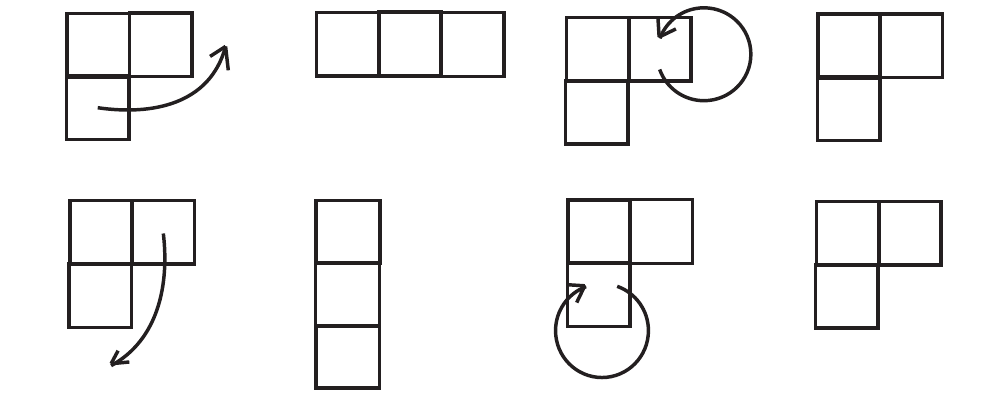}}
\end{center} 
\caption{$\Ind \circ \Res\  [2,1]$}
\label{figureindres}
\end{figure}

Let $V = V_{[n-1,1]}$. 
By the above Young rule, we have $\C^n = \Ind \un = \un + V$.
Using the classical formula $U \otimes \Ind W \simeq \Ind ((\Res U) \otimes W)$
we get, for all $U \in R(\SN)$,
$$
U + U \otimes V = U \otimes (\un + V) = \Ind \Res U
$$
i.e. $U \otimes V = (\Ind \Res U) - U$. Because of this, $N(\la)= N(V_{\la})$
can be determined combinatorially. First note that,
if $V_{\la} \into V^{\otimes (r-1)} \otimes V$,
then $V_{\la} \into V_{\mu} \otimes V$ for some irreducible
$V_{\mu} \into V^{\otimes (r-1)}$. An immediate consequence of the
above remarks is thus that the number $\la_1$ of boxes in the first row
for $\la$ 
satisfies
%is at least one less than the one for $\mu$, hence
$\la_1 \geq \mu_1 - 1$. By induction on $r$ this yields $r \geq n- \la_1$,
hence $N(\la) \geq n - \la_1$. One then easily gets
the following classical fact, for which we could not find an easy
reference.

%, but which is already implicit in Murnaghan's papers.
%Letting $V = V_{[n-1,1]}$, an immediate consequence of the above remarks is
%that, if $\la \into V^{\otimes r}$, then $n - \la_1 \leq r$ ; in particular,
%$N(\la) \geq n- \la_1$. One easily gets the following classical fact.

\begin{prop} \label{propniveau} For all $\la \vdash n$, we have $N(\la) = n- \la_1$.
\end{prop}
\begin{proof}
The proof is by induction on $r = n- \la_1$, the case $r = 0$ being
clear. Let $\la = [\la_1,\dots,\la_s]$ with 
$\la_1 \geq \la_2 \geq \dots \geq \la_s > 0$, $n - \la_1 = r+1$.
Since $n-\la_1 > 0$ we have $s \geq 2$. We
consider $\mu \vdash n$ defined by $\mu_1 = \la_1 + 1$,
$\mu_i = \la_i$ for $1 < i < s$, and $\mu_{s} = \la_s - 1$.
By the induction assumption, $N(\mu) = r$ and $V_{\mu} \into V^{\otimes r}$.
One of the components of $V_{\mu} \otimes V$ is $V_{\la}$ by the combinatorial
rule, hence $V_{\la} \into V_{\mu} \otimes V\into V^{\otimes (r+1)}$
and the conclusion follows by induction.
\end{proof}

For a partition $\la = [\la_1 , \la_2 , \dots]$ of $n$ with
$\la_i \geq \la_{i+1}$, we define the partition $\theta(\la) = 
[\la_2,\la_3,\dots]$ of $n-\la_1$. In diagrammatic terms, $\theta(\la)$
is the diagram deduced from $\la$ by \emph{deletion of the first row}
(see figure \ref{figdemo}).
Proposition \ref{propniveau} can thus be reformulated as
$$
|\theta(\la)| = N(\la).
$$

\section{Dvir's formula}

For three partitions $\la,\mu,\nu$ of arbitrary size,
we let $L_{\la,\mu,\nu}$ denote the Littlewood-Richardson
coefficient (see e.g. \cite{FH}). A remarkable discovery of Y. Dvir is that the
graded ring structure $(\gr R(\SN),\odot)$ is basically
given by these coefficients.

We first recall how to compute $L_{\la,\mu,\nu}$ with $|\nu| = |\la| + |\mu|$
using the Littlewood-Richardson rule : $L_{\la,\mu,\nu}$ is the
number of ways $\la$, as a Young diagram, can be expanded into
$\nu$ by using a \emph{$\mu$-expansion}. Letting $\mu = [\mu_1,
\dots,\mu_k]$, such a $\mu$-expansion is obtained by first adding
$\mu_1$ boxes labelled by $1$, then $\mu_2$ boxes labelled by $2$, and so on (that is,
at the $r$-th step we add $\mu_{r}$ boxes labelled $r$ to the preceedingly obtained diagram) 
so that
\begin{enumerate}
\item at each step, one still has a Young diagram
\item the labels strictly increase in each column
\item when the labels are listed from right to left in each row and
starting with the top row, we have the following property. For each $t \in [1, |\mu|]$,
the following holds : each label $p$ occurs at least as many times
as the label $p+1$ (when it exists) in the first $t$ entries.
\end{enumerate}
As an example, see figure \ref{figureexpansions} for the list of the $[2,2]$-expansions of
$[2,1,1]$ and figure \ref{figureexLR} for the two expansions leading to
$L_{[2,1],[2,1],[3,2,1]} = 2$. The reader can find in \cite{FH} other examples and
further details on this combinatorics.

\begin{figure}
\begin{center}
\includegraphics{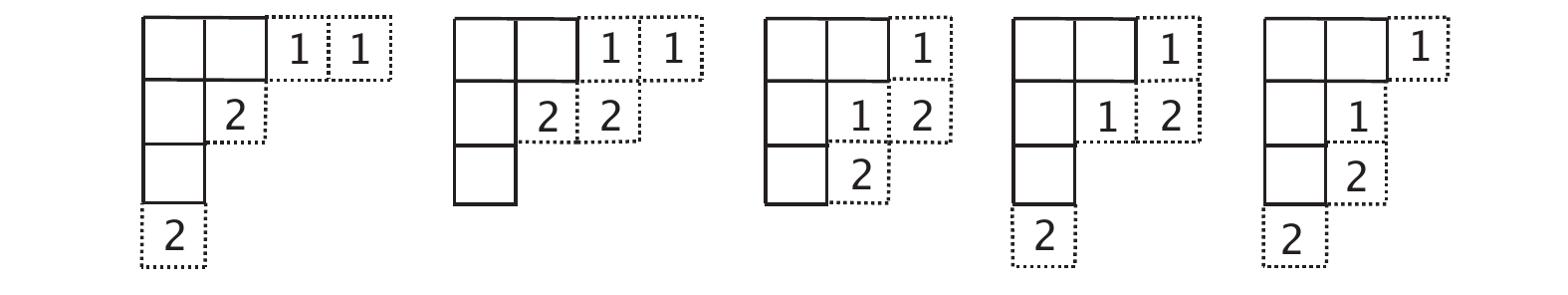}
\end{center} 
\caption{The $[2,2]$-expansions of
$[2,1,1]$}
\label{figureexpansions}
\end{figure}
\begin{figure}
\begin{center}
\includegraphics{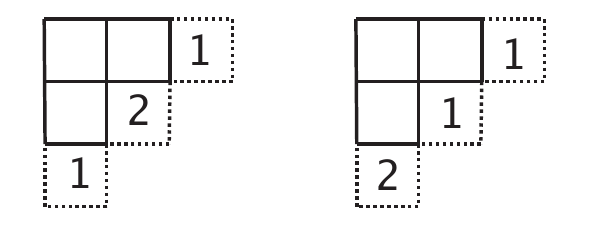}
\end{center} 
\caption{$L_{[2,1],[2,1],[3,2,1]} = 2$}
\label{figureexLR}
\end{figure}

For $\la, \mu,\nu \vdash n$, we let $C_{\la \mu \nu}$
denote the structure constants $V_{\la} \otimes V_{\mu} =
\sum_{\nu} C_{\la \mu \nu} V_{\nu}$ of $R(\SN)$. These constants, whose study has been
initiated by Murnaghan (1938), are notoriously complicated to understand.

For a partition $\la = [\la_1,\la_2,\dots]$ with $\la_i \geq \la_{i+1}$,
of $n$, define the partition $\theta(\la) = [\la_2,\la_3,\dots]$
of $n - \la_1$, and 
%note that 
let $d(\la) = |\theta(\la)| = \la_2 + \la_3 + \dots = n- \la_1$.
By proposition \ref{propniveau} above we have $d(\la) = N(\la) = \min \{ r \geq 0 \ | \ V_{\la} \into
V^{\otimes r} \}$. In particular $C_{\la,\mu,\nu} = 0$ whenever $d(\nu) > d(\la) + d(\mu)$.
Dvir's formula can be stated as follows

\begin{theo} (Dvir \cite{DVIR}, theorem 3.3) Let $\la,\mu,\nu$ be partitions of $n$ such that
$d(\la) + d(\mu) = d(\nu)$.
Then $C_{\la,\mu,\nu} = L_{\theta(\la),\theta(\mu),\theta(\nu)}$.
\end{theo}

In particular we get, inside $\gr R(\SN)$, the following formula :
$$
V_{\la} \odot V_{\mu} = \sum_{d(\nu) = d(\la) + d(\mu)} L_{\theta(\la),
\theta(\mu),\theta(\nu)} V_{\nu}.
$$

\section{The proof}

The main theorem is then an immediate consequence of the
following proposition. 
%Notice that $\Irr(\SN)$ also provides
%a basis for $\gr R(\SN)$ as a $\Z$-module. 
For the proof of this proposition, we will associate to
a Young diagram $\alpha = [\alpha_1,\alpha_2,\dots]$ its
interior $\alpha^{\circ}$ defined by the partition
$\alpha^{\circ}_i = \max(0,\alpha_i - 1)$, and its boundary
$\partial \alpha$ is
defined to be the ribbon made of the boxes in $\alpha$ which do
not belong to $\alpha^{\circ}$. The size $|\partial \alpha|$ of $\partial
\alpha$ (that is, its number of boxes) is clearly equal to
the number of rows in $\alpha$, or in other terms to the number of nonzero
parts of the partition $\alpha$.

\begin{prop} The ring $(\gr R(\SN),\odot)$ is generated by the $\Lambda^k V$,
$0 \leq k \leq n-1$.
\end{prop}
\begin{proof}
Recall that
$\Lambda^k V = V_{[n-k,1^k]}$, and note that $\theta([n-k,1^k]) = [1^k]$.
In particular $N(\Lambda^k V) = k$.
We identify each $V_{\la}$ with its image in $\gr R(\SN)$ and
let $Q$ denote the subring of $\gr R(\mathfrak{S}_n)$
generated by the $\Lambda^k V$. We prove that $V_{\la} \in
Q$ for all partition $\la$ of $n$ ($\la \vdash n$),
by induction on $d(\la) = |\theta(\la)|$. We have $d(\la) = 0
\Rightarrow \la = [n]  \Rightarrow V_{\la} = \Lambda^0 V$
and $d(\la) = 1
\Rightarrow \la = [n-1,1]  \Rightarrow V_{\la} = \Lambda^1 V$,
hence $V_{\la} \in Q$ if $d(\la) \leq 1$. We thus assume
$d(\la) \geq 2$ and that $V_{\mu} \in Q$ for all partitions $\mu$
with $d(\mu) < d(\la)$.

\begin{comment}
We need some more notation. For a partition $\alpha$ of $m$,
we let $\alpha'_1$ denote its number of nonzero parts, and
denote its size $m$ by $|\alpha|$. Define the partition
$\alpha^{\circ}$ by $\alpha^{\circ}_i = \max(0,\alpha_i - 1)$.
Clearly $|\alpha^{\circ}| = |\alpha| - \alpha'_1$ 
hence $|\alpha^{\circ}|\leq |\alpha|$, with equality only
if $\alpha = 0$.

We let $\alpha = \theta(\la)$ and use another induction on $d(\la) - \alpha'_1 = |\alpha| - \alpha'_1$.
The case $d(\la) - \alpha'_1 = 0$ means $V_{\la} = \Lambda^{\alpha'_1} V
\in Q$, so we can assume $d(\la) - \alpha'_1 \geq 1$.
\end{comment}

Letting $\alpha = \theta(\la)$ we use another induction on $|\alpha^{\circ}|$.
Note that $|\alpha^{\circ}| \leq |\alpha|$, with equality only if
$\alpha = \emptyset$. More generally, the case
$|\alpha^{\circ} | = 0$ means that
$V_{\la} = \Lambda^{|\partial \alpha|} V \in Q$, so we can assume
$|\alpha^{\circ}| \geq 1$.
 
\begin{comment}
Note that $d(\la) - \alpha'_1 = |\alpha| - \alpha'_1$.
The case $d(\la) - \alpha'_1 = 0$ means $V_{\la} = \Lambda^{\alpha'_1} V
\in Q$, so we can assume $d(\la) - \alpha'_1 \geq 1$.
\end{comment}

\begin{figure}
\begin{center}
\includegraphics{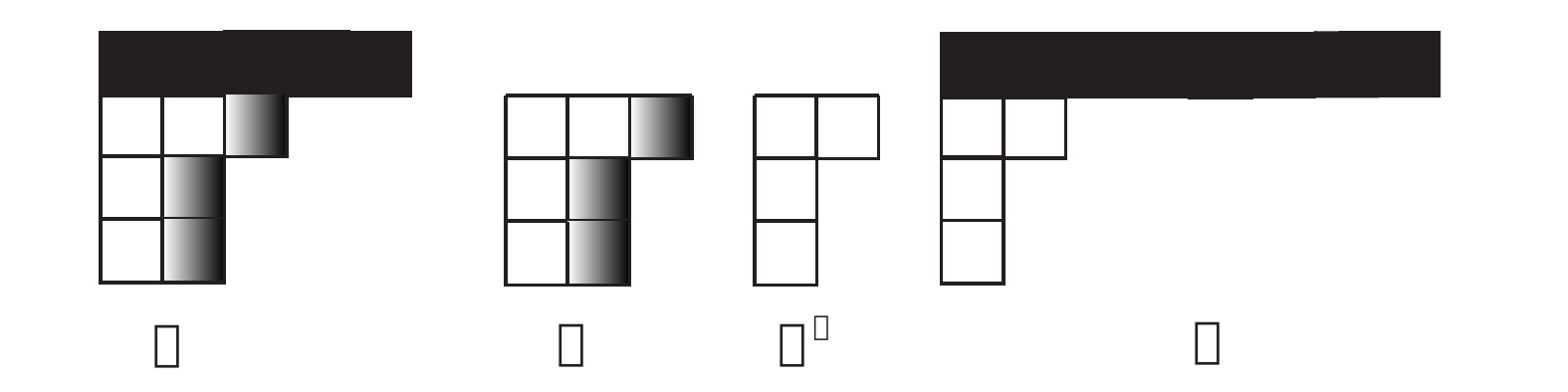}
\end{center} 
\caption{$\alpha = \theta(\la)$, $\alpha^{\circ}$ and $\mu$ for $\la = [5,3,2,2]$}
\label{figdemo}
\end{figure}

We let $r = |\partial \alpha| = |\alpha| - |\alpha^{\circ}|$.
Since $d(\la) \geq 2$ we have $\theta(\la) \neq 0$ and in particular $r \geq 1$.
Moreover $\la_1 = n- |\alpha| \geq \la_2$, hence $n - |\alpha^{\circ}| \geq \alpha_1 \geq \alpha^{\circ}_1$.
We thus can introduce the partition $\mu = [n-|\alpha^{\circ}|,\alpha^{\circ}_1,
\dots,]$ of $n$ (see figure \ref{figdemo} for an example) and consider $M = V_{\mu} \odot \Lambda^r V \in \gr R(\mathfrak{S}_n)$.
Since $|\alpha^{\circ}| < |\alpha|$ we have $d(\mu) < d(\la)$
hence $V_{\mu} \in Q$ by the first induction assumption so $M \in Q$.
Let $\nu \vdash n$
such that $M$ has nonzero coefficient on $V_{\nu}$.
We have $d(\nu) = d(\mu) + r = d(\la)$, hence
$\nu_1 = n-|\alpha| = \la_1$, and this coefficient is
$L_{\alpha^{\circ},[1^r],\theta(\nu)}$
by Dvir formula.

By the Littlewood-Richardson rule, this coefficient
$L_{\alpha^{\circ},[1^r],\theta(\nu)}$
is the number of ways that one can add boxes labelled $1,\dots,r$ on the
Young diagram of $\alpha^{\circ}$ with at most one box on each row
(with the graphic convention that $\alpha^{\circ}$
has $\alpha^{\circ}_i$ boxes on the $i$-th row),
the labels increasing along the rows,
and such that the augmented diagram corresponds to $\theta(\nu)$. We thus
clearly have $L_{\alpha^{\circ},[1^r],\alpha} = 1$, this corresponding to
adding a box marked $i$ on the $i$-th row for each $1 \leq i \leq r$.
Moreover,
if $L_{\alpha^{\circ},[1^r],\theta(\nu)}$
is nonzero, then either $\theta(\nu)$ has (strictly) more nonzero
parts than $\alpha$,
%$\theta(\nu)'_1 > \alpha'_1$, 
which means that
one box has been added to the empty $(r+1)$-st row, and in that
case we know that $V_{\nu} \in Q$ by the second
induction hypothesis 
%(as $\nu'_1 = 1+\theta(\nu)'_1$) ; 
%(as this means $|\partial \theta(\nu)| > r$, and $|\partial \nu| = 1+ | \partial \theta(\nu)|$ then implies $|partial \nu| > |\partial \la|$ hence) ; 
(as this means $|\partial \theta(\nu)| > r = |\partial \alpha|$, hence
$|\theta(\nu)^{\circ}| < |\alpha^{\circ}|$ since $|\alpha| = |\theta(\nu)|$); 
or, the $r$ boxes have been added to the first %$r$-th 
row,
which implies $\theta(\nu) = \alpha$ hence $\nu = \la$. We thus get $M \equiv V_{\la}$ modulo $Q$,
$V_{\la} \in Q$ and the conclusion follows by induction.
\end{proof}

A careful look at the above proof shows that we proved
a more technical but also more precise result. For $\la,\mu \in \mathfrak{S}_n$,
we define $\la \prec \mu$ if either $N(\la) < N(\mu)$,
or $N(\la) = N(\mu)$ and $|\theta(\la)^{\circ}| < |\theta(\mu)^{\circ}|$,
and we denote by $R_{\la}$ (resp. $\overline{R}_{\la}$) the
$\Z$-submodule of $R(\mathfrak{S}_n)$ (resp. $\gr R(\mathfrak{S}_n)$)
spanned by the $\kappa \in \Irr(\mathfrak{S}_n)$ with $\kappa \prec \lambda$.
The above proof actually shows the following.

\begin{prop} For every $\la \in \Irr(\mathfrak{S}_n) \setminus \{ \un \}$,
there exists $\hat{\la} \in \Irr(\mathfrak{S}_n)$
with $\hat{\la} \prec \la$ and $k \in \Z_{\geq 0}$ such that
$\hat{\la} \odot \Lambda^k V \in \la + \overline{R}_{\la}$.
\end{prop}

Since $\mathcal{F}_{N(\kappa)-1}(\mathfrak{S}_n) \subset R_{\kappa}$ this
immediately implies
\begin{cor} \label{corprod}
For every $\la \in \Irr(\mathfrak{S}_n) \setminus \{ \un \}$,
there exists $\hat{\la} \in \Irr(\mathfrak{S}_n)$
with $\hat{\la} \prec \la$ and $k \in \Z_{\geq 0}$ such that
$\hat{\la} \otimes \Lambda^k V \in \la + R_{\la}$.
\end{cor}

\section{An application}

One can use this result to give a proof of the well-known fact
that all complex linear representations of the symmetric group
can actually be realized over $\Q$. We first recall the following
lemma.

\begin{lemma} \label{lemmult1} Let $G$ be a finite group,
$\k$ a number field, $\rho : G \to \GL_N(\k)$
a linear representation of $G$ defined over $\k$,
and $\rho_{\C} : G \to \GL_N(\C)$ its complexification. If
$\varphi$ is an irreducible subrepresentation of $\rho_{\C}$
occuring with multiplicity one whose character takes values
in $\k$, then $\varphi$ can be realized over $\k$.
\end{lemma}
\begin{proof} This is an immediate consequence of the fact that the projection on the
$\varphi$-isotopic component of $\rho_{\C}$ is
given by $\frac{\dim \varphi}{|G|} \sum_{g \in G} \overline{\chi(g)}
\rho(g)$ (see e.g. \cite{FH} (2.32)), which is an endomorphism of $\k^N$
under our assumptions.
\end{proof}

We now can deduce the following well-known result.

\begin{theo} \label{theorealQ} Every complex linear representation of $\mathfrak{S}_n$
can be realized over $\Q$.
\end{theo}
\begin{proof}
We use first that the natural permutation module $\C^n$ is obviously
realizable over $\Q$, and that $\C^n = \un + V$. This implies
that the character associated to $V$ is defined over $\Q$, hence
$V$ can be realized over $\Q$ by lemma \ref{lemmult1} (or, directly,
$V$ can be identified to the rational subspace $\{ (x_1,\dots,x_n) \in
\Q^n \ | \ x_1 + \dots + x_n = 0 \}$). It follows that all the
$\Lambda^r V$ can be realized over $\Q$. 

Since $\Irr(\mathfrak{S}_n)$ is clearly a well-founded set under $\prec$,
with minimal element $\un$,
%and that all descending chain for $\prec$ abuts to some $\Lambda^k V$
one can now use this relation to prove our statement by
induction.

Let $\la \in \Irr(\mathfrak{S}_n)$.
Corollary \ref{corprod} implies
that there exists $\hat{\la} \prec \la$ and $k \in \Z_{\geq 0}$ such that
$M = \hat{\la} \otimes \Lambda^k$, which is realizable over $\Q$ by
our induction assumption, contains $\la$ with multiplicity 1, and has the
property that the quotient representation $M/\la$ is also realizable over $\Q$
by the same induction assumption. This proves that the character
of $\la$ takes values in $\Q$, and then that $\la$ is realizable over
$\Q$ by lemma \ref{lemmult1}. This concludes the proof.

\begin{comment}
Now consider
an arbitrary $\la \in \Irr(\mathfrak{S}_n)$. Theorem \ref{maintheo}
says that, in $R(\mathfrak{S}_n)$, we have $\la =
\sum_m c_m x_m$ for some $c_m \in \Z$ and $x_m$
the tensor product of a collection of $\Lambda^r V$. In particular,
$x_m$ is a representation of $\mathfrak{S}_n$ that
can be realized over $\Q$. Rewriting this equation as
$\la + \sum_{c_m < 0} (- c_m) x_m = \sum_{c_m > 0} c_m x_m$,
we see that the right-hand side is a representation realizable
over $\Q$ that contains 
BBBBIIIIIIIIPPPPPPPPPPP Probleme
\end{comment}

\end{proof}

\section{Generalization attempts}

The symmetric group is an irreducible complex (pseudo-)reflection group.
Recall that such a group is a finite subgroup $W$ of $\GL(V)$
for $V$ some finite-dimensional complex vector
space acted upon irreductibly by $W$, with $W$ generated by its
reflections, namely elements of $\GL(V)$ which fix a hyperplane.
The dimension of $V$ is called the \emph{rank} of $W$.
%
%For such a
%group $W < \GL(V)$ for $V$ some finite-dimensional complex vector
%space acted upon irreductibly by $W$, 

For such a group, it is a classical
result of Steinberg that the representations $\Lambda^k V$
are irreducible (see e.g. \cite{BOURBAKI} ch. 5 \S 2 exercice 2),
and are thus natural generalization of \emph{hooks}.
%generalizing the corresponding fact for $\mathfrak{S}_n$
%The dimension of $V$ is called the \emph{rank} of $W$.

Among other similarities, theorem \ref{theorealQ} admits
a natural generalization to these groups. Indeed, it can first be
shown that the representation $V$ can be realized over its character field $\k$
(i.e. the number field generated by the values taken by its character),
sometimes called its \emph{field of definition}. Moreover, it is a theorem
of M. Benard that \emph{every} representation of $W$ can
be realized over $\Q$ (see \cite{BENARD}, and also \cite{BESSIS}, \cite{MARINMICHEL}
for other proofs), thus providing a complete generalization of theorem
\ref{theorealQ}. We now investigate to what extent theorem \ref{maintheo}
could be generalized.

The irreducible complex reflection groups have been classified
by Shephard and Todd (see \cite{ST}). There is an infinite series $G(de,e,r)$
depending on three integral parameters $d,e,r$, plus 34 exceptions
$G_4,\dots,G_{37}$. For the representation theory of the $G(de,e,r)$
we refer to \cite{ARIKI}. 

Note that, for a given group with known character table, it is easy to
check by computer whether a given subset $\mathcal{B}$ of $\Irr(W)$ generates $R(W)$.
Indeed, the ring $R(W) = \Z \Irr(W)$ is a free $\Z$-module with basis
$\Irr(W)$ ; assume we are given a subset $\mathcal{B} \subset \Irr(W)$
with $\un \in \mathcal{B}$, and let $A$ denote the subring
of $R(W)$ generated by $\mathcal{B}$. The embedding
$R(W) \subset \End_{\Z} R(W) \simeq \End_{\Z} (\Z\Irr(W))$
identifies $A$ with the minimal $\Z$-submodule of $\Z \Irr(W)$
containing $\un_W$ which is stable under multiplication by $\mathcal{B}$.
This identifies $a \in A$ with $a.\un \in \Z \Irr(W)$.
Starting with the $\Z$-module $A_0 = \Z \un$ of rank 1,
multiplication by the elements of $\mathcal{B}$ iteratively
provides a sequence of submodules $A_0 \subset
A_1 \subset \dots$ which eventually stops at $A_{\infty} = A$
by noetherianity of the $\Z$-module $R(W)$.

%Iterating kronecker products by the characters in $\mathcal{B}$
%until this module stops increasing provides
%an algorithm that determines $A$.

If $W$ has rank 2, we are able to prove case-by-case the
following.

\begin{prop} If $W$ is an irreducible complex reflection group of rank 2, then
$R(W)$ is generated by $V$ and the 1-dimensional representations.
\end{prop}
\begin{proof}
The case of exceptional reflection groups is checked by computer, using the
algorithm above. The non-exceptional ones are the $G(de,e,2)$, so we assume
$W = G(de,e,2)$. The irreducible representations of $W$ have dimension at most 2. The ones
of dimension 2 can be extended to $G(de,1,2)$, so we can assume
without loss of generality that $e=1$.
The group $W$ is generated by $t = \mathrm{diag}(1, \zeta)$ with $\zeta = \exp(2 \ii \pi/d)$ and $s$ the
permutation matrix $(1\ 2)$. Its two-dimensional
representations are indexed by couples $(i,j)$ with $0 \leq i < j < d$.
We extend this notation to $i,j \in \Z$ with $j \not\equiv i \mod d$ by taking representatives modulo
$d$ and letting $(i,j) = (j,i)$.
 A matrix model for the images of $t$ and $s$ in the
representation $(r,r+k)$  is
$$
t \mapsto \begin{pmatrix} \zeta^r & 0 \\ 0 & \zeta^{r+k} \end{pmatrix}\ \ 
s \mapsto \begin{pmatrix} 0 & 1 \\ 1 & 0 \end{pmatrix}
$$
In particular, $V = (0,1)$. From these explicit models it is straightforward to check that
$(0,1) \otimes (0,1)$ is the sum of $(0,2)$ and 1-dimensional representations,
and that $(0,1) \otimes (0,k) = (0,k+1) + (1,k)$.
Then we consider the 1-dimensional representation $\chi_1 : t \mapsto \zeta, s \mapsto 1$.
It is clear that $(i,j) \otimes \chi_1 = (i+1,j+1)$. Letting $Q$ denote
the subring of $R(W)$ generated by $V$ and the 1-dimensional representations,
through tensoring by $\chi_1$ is it enough to show that $(0,k) \in Q$ for all
$1 \leq k \leq d$. By definition $(0,1) \in Q$, tensoring by $(0,1)$
yields $(0,2) \in Q$, and finally $(0,1) \otimes (0,k) = (0,k+1) + \chi_1 \otimes (0,k-1)$
proves the result by induction on $k$.
\end{proof}

Among the higher rank exceptional groups, we check by computer
that the union of the $\Lambda^k V$ and the one-dimensional representations generates $R(W)$
exactly for the groups $G_{23} = H_3, G_{24},G_{25},G_{26},G_{30} = H_4, G_{33}, G_{35} = E_6$
(but not $E_7$ nor $E_8$ !).

In the more classical case of the Coxeter groups $W$ of type $B_n$ and $D_n$,
it is easily checked that the subring generated by the $\Lambda^k V$
has not full rank in $R(W)$ (for $n \geq 4$).
It is thus natural to consider the non-faithful reflection representations
$U$ of dimension $n-1$ of these groups,
which correspond to $([n-1,1],\emptyset)$ and $\{ [n-1,1],\emptyset \}$
in the usual parametrizations of their irreducible representations
(see \cite{GP}). These are deduced from $V_{[n-1,1]} \in \Irr(\SN)$
through a natural morphism $W \onto \SN$. A computer
check for small values of $n$ motivates the following conjecture.

\begin{conj} For $W$ a Coxeter group of type $B_n$ or $D_{2n+1}$, $R(W)$ is generated
by the $\Lambda^k V, \Lambda^k U, k \geq 0$.
\end{conj}

The proof of such a conjecture would probably involve an understanding
of the structure constants in $R(W)$ comparable to
Dvir's formula for $\SN$. Unfortunately, the combinatorial
study of the representation ring of these more general Coxeter
groups seems to be only at the beginning.

For a group of type $D_{2n}$, it can be checked that the subring generated by such elements has smaller
rank already for $D_4$. This is a general phenomenon, as can be seen in the
following way. Recall that a group $W$ of type $D_n$ is an index
2 subgroup of a Coxeter group $\widetilde{W}$ of type $B_n$.
By Clifford theory, an irreducible representations of $\widetilde{W}$
parametrized by $(\la,\mu)$ with $|\la|+|\mu| = n$
restricts either to an irreducible representation $\{ \la, \mu \}$
of $W$, precisely in the case $\la \neq \mu$, or,
in the case $\la = \mu$, to a direct sum of two irreducibles
usually denoted $\la^+$ and $\la^-$. Note that such $\la^{\pm}$
exist if and only if $n$ is even.

Choosing some $s \in \widetilde{W} \setminus W$ and
letting $\Ad s : x \mapsto sxs^{-1}$ be the automorphism
of $W$ induced by $s$, the map $\rho \mapsto \rho \circ \Ad s$
induces a $\Z$-linear involution $\eta$ of $R(W)$ which fixes
the $\{ \la, \mu \}$ and maps $\la^{\pm}$ to $\la^{\mp}$. Letting
$R(W)^{\eta}$ denote the invariant subspace, we have $R(W)^{\eta} = R(W)$
if and only if $n$ is odd. Clearly the $\Lambda^k V$ and $\Lambda^k U$
are always fixed by $\eta$, and this explains why the
subring they generate cannot be $R(W)$ when $R(W)^{\eta} \neq R(W)$.
We do not have any serious guess for a
natural generating set in these cases.

\end{document}